\documentclass[]{article}
\usepackage{amssymb}
\usepackage{amsmath}
\usepackage{amsthm}
\usepackage{color}
\usepackage[
top    = 3.00cm,
bottom = 3.00cm,
left   = 3.00cm,
right  = 2.80cm]{geometry}
\usepackage[style=alphabetic,natbib=true,bibencoding=ascii,backend=bibtex]{biblatex}
\usepackage{hyperref}
\usepackage{etoolbox}
\usepackage{mathtools}
\usepackage{tikz-cd}
\usepackage{faktor}
\usepackage{xfrac} 
\apptocmd{\thebibliography}{\raggedright}{}{}

\usepackage{titletoc}

\dottedcontents{section}[1.5em]{\bfseries}{1.3em}{.6em}

\newtheorem{theorem}{Theorem}[section]
\newtheorem{lemma}[theorem]{Lemma}
\newtheorem{proposition}[theorem]{Proposition}

\newtheorem{claim}[theorem]{Claim}

\newtheorem{non-example}{Non-example}[section]
\title{On the torsion in the cohomology of the integral structure sheaf of affinoid adic spaces}
\author{Emiliano Torti \\ Universit\'e de la Polyn\'esie Fran\c{c}aise\\ torti.emiliano@gmail.com}

\begin{document}

\maketitle

\begin{abstract}\noindent
We prove that the cohomology of the integral structure sheaf of a normal affinoid adic space over a non-archimedean field of characteristic zero is uniformly torsion. This result originated from a remark of Bartenwerfer around the 1980s and it partially answers a recent question of Hansen and Kedlaya (see also Problems 27 and 39 in the Non-Archimedean Scottish Book). 
\end{abstract}


\begin{section}{Introduction}
Let $\mathbb{K}$ be a complete non-archimedean field of characteristic zero and denote by $\mathbb{K}^\circ$ (resp. $\mathfrak{m}$) the subring of integral elements (resp. the maximal ideal of topologically nilpotent elements). Let $(A, A^+)$ be sheafy a uniform Tate (complete) Huber pair over $\mathbb{K}$ and denote by $I$ the ideal of definition of $A^+$. Let $X:=\text{Spa}(A, A^+ )$ be its attached affinoid adic space over $\mathbb{K}$. Denote by $\mathcal{O}_X$ (resp. $\mathcal{O}_X^+$ ) its structure sheaf (resp. integral structure subsheaf). Kedlaya and Liu (see Theorem 2.4.23 in \cite{KL15}) showed that the sheaf cohomology of the structure sheaf satisfies $H^q (X , \mathcal{O}_X )=0$ for all $q\geq 1$, i.e. the structure sheaf $\mathcal{O}_X$ is acyclic. This property however does not hold in general if we replace $\mathcal{O}_X$ by the integral structure subsheaf $\mathcal{O}_X^+$, see the example 6.6 in \cite{HK25}. In accordance with Hansen and Kedlaya (see \cite{HK25}), we say that an $A^+$-module $M$ is \textit{torsion} (resp. \textit{uniformly torsion}) if every element is killed by a power of $I$ (resp. $I^n M$=0 for some $n\in \mathbb{Z}_{\geq 1}$). Moreover, following Gabber and Ramero (see \cite{GR03}), we say that a $\mathbb{K}^\circ$-module $M$ is \textit{almost zero} if $\mathfrak{m} M=0$. The following result, which summarizes the state of the art on the known properties of the groups $H^* (X , \mathcal{O}_X^+ )$, is the starting point of our investigation:
\begin{theorem}[Bartenwerfer, Luktebohmert, Hansen-Kedlaya]
Let $(A, A^+)$ be a affinoid algebra topologically of finite type over a non-archimedean field $\mathbb{K}$ of characteristic zero, and let $X=\text{Spa}(A, A^+ )$ its attached affinoid adic space. Then we have the following: \\
(I) $A$ is semi-normal  if and only $H^1 (X, \mathcal{O}_X^+)$ is uniformly torsion,\\
(II) if $A$ is smooth then $H^q (X, \mathcal{O}_X^+ )$ is uniformly torsion for all $q>0$,\\
(III) if $X$ admits a smooth formal model, then $H^q ( X , \mathcal{O}_X^+ )=0 $ for all $q>0$.
\end{theorem}
\noindent
While the statements (II) and (III) were known several years ago thanks to Bartenwerfer, Luktebohmert et al. (see for example Folgerung 3 in \cite{Bar79}, (for $q$=$1$) Theorem 2 in \cite{Bar78} and also \cite{Luk22}, \cite{KST19}), the statement $(I)$ is one of the results of recent remarkable investigations of Hansen and Kedlaya (see Theorem 10.3 in \cite{HK25}). They consider the property of having uniformly torsion integral cohomology in order to define certain subclasses of Huber pairs (called plus-sheafy and diamantine) which plays a central role in some geometric constructions in adic geometry (see \cite{HK25}).\\
In the second half of the 1970s, Bartenwerfer (see \cite{Bar78}) suggests that it might be possible to relax the smoothness hypothesis to normality in order to show that the 1st cohomology group of the integral structure sheaf is killed by a non-zero element in $\mathbb{K}^\circ$, (in the original text in german, see pag. 2 in \cite{Bar78}): 
\textit{"Mit einem Riemannschen Hebbarkeitsargument bzw. Kugelsatz diirfte sich die Voraussetzung "glatt " leicht zu "absolut normal" abschwachen lassen"}.\\
Note that in modern literature one still refers to the extension of $p$-adic analytic functions as Riemann's Hebbarkeissatz which has been proved by Bartenwerfer in the rigid analytic context as we will see later on. Moreover, for example, a version of Riemann's Hebbarkeissatz has also been proved for normal perfectoid spaces by Scholze (see sec. II.3 in \cite{Sch15}). Coming back to the problem of understanding the torsion in the integral structure sheaf cohomology, following Hansen and Kedlaya (see \cite{HK25}) one can try to be more precise and ask the natural question whether it is possible to identify the properties which $A$ has to satisfy which would be equivalent to $H^* (X , \mathcal{O}_X^+ )$ being uniformly torsion (see Question 6.15 in \cite{HK25} and Problems 27 and 39 in the Non-Archimedean Scottish Book). To be precise, we define for each positive integer $N\in \mathbb{Z}_{\geq 1}$, the following property $(P_N)$ for the Huber pair $(A; A^+)$:  the cohomology groups $H^q (X , \mathcal{O}_X^+ )$ are uniformly torsion for $1 \leq q \leq N$. Note that by De Jong and Van Der Put (see Prop. 2.5.8 in \cite{DJVP96}), because the space $X$ is in particular quasi-compact and of dimension $d$, the only relevant cohomology is the one in degrees between $1$ and $d$. In this terms, Hansen and Kedlaya showed that $(P_1 )$ is equivalent to be seminormal and Bartenwerfer showed that if $A$ is smooth than it satisfies $(P_\infty)$. Note moreover that the class of Huber pairs which satisfy $(P_\infty)$ is pretty large. For example, other than smooth affinoid algebra, it also contains all perfectoid affinoid Huber pairs as Scholze proved (see Theorem 1.8 in \cite{Sch12}). This fact and much more information concerning the relationship among these classes of Huber pairs can be found in a nice diagram at the end of \cite{HK25}. 
Motivated by Bartenwerfer's remark, the main goal of this article is to show that if $A$ is normal than it satisfies as well $(P_\infty)$, i.e. we have the following:
\begin{theorem}
Let $X=\text{Spa} (A, A^+ )$ be a normal affinoid adic space topologically of finite type. \\
Then we have that:
$$H^q (X, \mathcal{O}_X^+ ) \text{ is uniformly torsion for all }q\geq 1.$$
In particular, for each $q\geq 1$ there exists $N\in \mathbb{Z}_{\geq 1}$ such that $\mathfrak{m}^N H^q (X, \mathcal{O}_X^+ ) =0$. 
\end{theorem}
\noindent
Generally speaking, proving $(P_\infty)$ in the smooth case has the advantage that one can essentially reduce the problem to the case where $X$ admits a finite \'etale Noether normalization map to $\mathbb{B}^d$ and then use such finite \'etale map to transport the information from the target since it is well-known (thanks to Bartenwerfer in \cite{Bar78}, \cite{Bar79}) that $H^q (\mathbb{B}^d , \mathcal{O}_{\mathbb{B}^d}^\circ )=0$ for $q \geq 1$. However, such strategy doesn't hold if we weaken the assumption on $X$ and assume only normality rather than smoothness. In order to correct such approach, we combine the study of covers in the sense of Hansen with Riemann's Hebbarkeitssatz in the $p$-analytic context. As a side note, it is worth to mention that the question is also very interesting if we remove the hypothesis that the characteristic of $\mathbb{K}$ is zero. Indeed, for example Scholze results still holds for affinoid perfectoid spaces over a perfectoid field of characteristic $p$. However, in the case of affinoid algebras topologically of finite type, the "\'etale approximation" approach has several subtle issues. We hope to investigate this line of research in the future.\\
Concerning terminology, we freely refer to the rigid setting by relying on the equivalence of categories between quasi-separated rigid analytic varieties and quasi-separated adic spaces locally of finite type established by Huber (see Prop. 4.5 in \cite{Hub94}).\\
\break
\textbf{Acknowledgments: } I would like to thank A. Conti for numerous useful conversations on the topic. I would also like to thank G. Bisson and A. Rahm in the GAATI group at University of French Polynesia who supported my research as part of the MELODIA project. The MELODIA project is funded by ANR under grant number ANR-20-CE40-0013.
\end{section}
\newpage
\begin{section}{On the torsion of the integral structure cohomology via Riemann's Hebbarkeitssatz}
This section is dedicated to the proof of the main result of the article:
\begin{theorem}
Let $X=\text{Spa} (A, A^+ )$ be a normal affinoid adic space topologically of finite type. \\
Then we have that:
$$H^q (X, \mathcal{O}_X^+ ) \text{ is uniformly torsion for all }q\geq 1.$$
In particular, for each $q\geq 1$ there exists $N\in \mathbb{Z}_{\geq 1}$ such that $\mathfrak{m}^N H^q (X, \mathcal{O}_X^+ ) =0$. 
\end{theorem}
\noindent
We start by observing that by the Noether's Normalization Theorem (see for example Corollary 6.1.2/2 in \cite{BGR84}), there exists a finite surjective morphism $\psi: X \rightarrow \mathbb{B}^d =\text{Spa}(T_d , T_d^\circ )$. Since we are dealing with normal (in particular, reduced) affinoid adic spaces topologically of finite type and we are safe to assume that $A^+ =A^\circ$ (see Lemma. 4.4 in \cite{Hub94}), we can considered fixed (from a formal schemes perspective) the formal models $\text{Spf}(A^\circ )$ and $\text{Spf}(T_d^\circ )$. We have the following (which applies in our "Noether Normalization" case): 
\begin{proposition}\label{restriction}
Let $\text{Spf}(A)$ and $\text{Spf}(B)$ be formal schemes and let $X$ and $Y$ be their respective generic fiber. Suppose there is a finite injective homomorphism $\varphi : A \rightarrow B$, inducing a morphism $\varphi^{\text{rig}} : Y \rightarrow X$. Assume also that $X$ and $Y$ are normal. Then there is an analytically closed, nowhere dense subset $Z \subset X$ such that the morphism 
$$\varphi^{\text{rig}} : Y - (\varphi^{\text{rig}})^{-1} (Z) \rightarrow X - Z$$
is finite and flat. Moreover, the subset $Z$ can be chosen such that $\text{codim}_X (Z) \geq 2$. 
\end{proposition}
\begin{proof}
This is Lemma 7.3.2 in \cite{DJ95}. 
\end{proof}
\noindent
This sort of generic flatness results are well-known and some slightly stronger result can be found in Bhatt and Hansen (see \cite{BH22}). Note that in our special case, we can say something stronger. Indeed, since $\text{char}(\mathbb{K})=0$, we have that $\psi$ is generically \'etale, i.e. there exists a Zariski-open $U \subset \mathbb{B}^d$ such that $\psi: \psi^{-1} (U) \rightarrow U $ is flat and unramified (see for example sec. 6.2 in \cite{DJ95}). Note that this is very special in the case of characteristic zero and the situation is much more subtle if the characteristic is positive (see \cite{BH22}).
Now, in order to use this sort of etale approximation of $\psi$ to compute the integral cohomology, we rely on Bartenwerfer's rigid analytic version of Riemann's Hebbarkeitssatz (see the original sec. 3 in \cite{Bar76}, see Thm.2.6 in \cite{Han20} for an english version):
\begin{proposition}\label{hartog}
Let $X$ be a normal rigid space and let $Z\subset X$ be a nowhere dense closed analytic subset with $j : X-Z \rightarrow X$ the inclusion of the open complement. Then we have isomorphism of sheaves $\mathcal{O}_X^+ \cong j_* \mathcal{O}_{X-Z}^+$  and $\mathcal{O}_X \cong (j_* \mathcal{O}_{X-Z}^+ )[\frac{1}{\omega}]$. In particular, if $X$ is affinoid we have $\mathcal{O}_X (X)\cong \mathcal{O}_X^+ (X-Z )[\frac{1}{\omega}]$.
\end{proposition}
\noindent
Before proceeding, we now introduce a useful terminology from \cite{Han20}. Following Hansen (see Def. 1.5 in \cite{Han20}), let $X$ be a normal rigid space. A cover of $X$ is a finite surjective map $\varphi: Y \rightarrow X$ from a normal rigid space $Y$, such that there exists some closed nowhere-dense analytic subset $Z \subset X$ with $\varphi^{-1} (Z)$ nowhere dense and such that $Y - \varphi^{-1}(Z) \rightarrow X - Z$ is finite and \'etale. \\
We remark that in our situation, Riemann's Hebbarkeitssatz can be simultaneously applied to both $X$ and $\mathbb{B}^d$ with respect to the map $\psi$. To be precise, we have the following (see Prop. 2.8 in \cite{Han20}):
\begin{proposition}
Let $X$ be a normal rigid space and let $\pi : Y \rightarrow X$ be a cover for $X$. Then each irreducible component of $Y$ maps surjectively onto some irreducible component of $X$. Moreover, if $V\subset X$ is any closed nowhere-dense analytic subset, then $\pi^{-1} (V)$ is nowhere-dense. 
\end{proposition}
\noindent
So until now we have shown that every normal affinoid adic space topologically of finite type has a large (i.e. except a nowhere dense closed analytic subset) finite \'etale Noether Normalization. Now, in order to prove that the integral cohomology groups are uniformly torsion we proceed as follows. Since $Z$ is a closed analytic set, we can write it as $Z=V(I)$ where $I$ is a finitely generated ideal inside $\mathcal{O}_X (X)$. Without loss of generality, we can reduce the problem in assuming that $I$ is a principal ideal. Indeed, by decomposing $I$ as a sum of $I_{p}$ and $\tilde I$ where $I_p$ is principal, i.e. $I=I_p +\tilde I$, we have the following two relations:
$$X - V(I) = (X- V(I_p )) \cup (X- V(\tilde I ))$$
$$(X-V( I_p )) \cap (X- V(\tilde I ))= X-V (I_p \cdot \tilde I ).$$
By an induction argument on the number of generators of $I$ and thanks to the Mayer-Vietoris sequence, we deduce that it is enough to prove that the cohomology of the integral structure sheaf is uniformly torsion under the condition that $I$ is principal, i.e. $I=( z )$ for some element $z \in \mathcal{O}_X^\circ (X)$. We now have the following useful result (it is an extension of Lemma A.10 in \cite{Luk22}):
\begin{lemma}
Let $X=\text{Spa}(A, A^\circ)$ be a $d$-dimensional normal affinoid adic space topologically of finite type over a non-archimedean field of characteristic 0. Assume that $\psi: X\rightarrow \mathbb{B}^d$ is a finite cover of $\mathbb{B}^d$ (in the sense of Hansen) which is \'etale over $\mathbb{B}^d - V(z)$ for some non-zero section $z\in T_d^\circ$ and assume that $\psi$ is of rank $n\in \mathbb{Z}_{\geq 1}$. Then we have that: \\
\break
$(i)$ $\psi_*  \mathcal{O}_X$ is a finite and free coherent $\mathcal{O}_{\mathbb{B}^d}$-module of rank $n$,\\
\break
$(ii)$ there exists $\alpha \in \mathbb{Z}_{\geq 1}$ such that:
$$ z^\alpha \psi_*  \mathcal{O}^\circ_X \subset \oplus_{i=1}^{n}  {\mathcal{O}^\circ}_{\mathbb{B}^d} e_i \subset 
\oplus_{i=1}^{n} \mathcal{O}_{\mathbb{B}^d} e_i = \psi_* \mathcal{O}_X ,$$
for linearly independent elements $e_1 , \dots , e_n \in A^\circ$. 
\end{lemma}
\begin{proof} 
Concerning $(i)$, denote for simplicity $V(z)=V$. Since $\psi$ is a finite affinoid morphism, we have that $\psi_* \mathcal{O}_X$ is a coherent $\mathcal{O}_{\mathbb{B}^d} $-module. Let $i$ and $j$ respectively be the immersions $i: \mathbb{B}^d - V \rightarrow \mathbb{B}^d$ and $j: X - \psi^{-1}(V) \rightarrow X$. Here, we rely heavly on the extension result given by Prop. \ref{hartog} which ensures us that $\mathcal{O}_X \cong j_* \mathcal{O}_{X- \psi^{-1}(V)}$ and $\mathcal{O}_{\mathbb{B}^d} \cong i_* \mathcal{O}_{\mathbb{B}^d- V}$ (and similarly for the integral subsheaves). Hence, we have the following chain of isomorphisms of $\mathcal{O}_{\mathbb{B}^d}$-modules $\psi_*  \mathcal{O}_X \cong \psi_* (j_*\mathcal{O}_{X- \psi^{-1}(V)} )\cong i_* (\psi_* \mathcal{O}_{X- \psi^{-1}(V)})$ since $\psi \circ j = i \circ \psi$. 
Note that by Prop \ref{restriction}, the map $\psi: X - \psi^{-1}(V) \rightarrow \mathbb{B}^d - V$ is finite and \'etale. As a consequence, we have that $\psi_* \mathcal{O}_{X- \psi^{-1}(V)}$ is a locally free $ \mathcal{O}_{\mathbb{B}^d- V}$-module. We conclude that $i_* (\psi_* \mathcal{O}_{X- \psi^{-1}(V)})$ is a locally free $ i_* (\mathcal{O}_{\mathbb{B}^d- V})$-module.  We deduce that $\psi_*  \mathcal{O}_X$ is a locally free coherent $\mathcal{O}_{\mathbb{B}^d}$-module. A result of Luktebohmert (see Satz 1 and following corollary in \cite{Luk77}), ensures us that every locally free modules over $T_d$ is also free. This completes the proof of part $(i)$.\\ 
Now, we focus on part $(ii)$. We know that $\psi_* \mathcal{O}_X$ is a free $\mathcal{O}_{\mathbb{B}^d }$-module. Fix a basis $e_1 , \dots, e_n \in \mathcal{O}_X^\circ$ such that $\psi_* \mathcal{O}_X \cong \oplus_{i=1}^{n} \mathcal{O}_{\mathbb{B}^d} e_i$. \\
The claim is to prove that there exists an element $\alpha \in \mathbb{Z}_{\geq 1}$ such that:
$$z^\alpha \psi_*  \mathcal{O}^\circ_X \subset \oplus_{i=1}^{n} {\mathcal{O}}^\circ_{\mathbb{B}^d} e_i \subset 
\oplus_{i=1}^{n} \mathcal{O}_{\mathbb{B}^d} e_i = \psi_* \mathcal{O}_X ,$$
for linearly independent elements $e_1 , \dots , e_n \in A^\circ$. 
Note first that since $\psi$ finite, the sheaf $\psi_* \mathcal{O}^\circ_X$ is finitely generated as  $\mathcal{O}^\circ_{\mathbb{B}^d}$-module. So in order to get the claim it is sufficient to show that once a generator is fixed, say $\psi_* f \in \psi_* \mathcal{O}^\circ_X$, we have that there exists an $\alpha \in \mathbb{Z}_{\geq 1}$ and $b_1 , \dots , b_n\in \mathcal{O}^\circ_{\mathbb{B}^d}$ such that $z^\alpha\psi_* f = \sum_i b_i e_i$ for some positive integer $\alpha$. Since there are only a finite number of generators than one can proceed in taking the maximum among the $\alpha$'s. \\
Define the open $U:= \mathbb{B}^d - V(z)$ and fix a global section $f\in \mathcal{O}_X (X)\cong A$ as above. In order to simplify the notation set $N:=\psi_* \mathcal{O}^\circ_X$ (resp. $M:=\psi_* \mathcal{O}_X$) as finitely generated coherent $\mathcal{O}^\circ_X$-module (resp. finite and free $\mathcal{O}_X$-module of rank $n$). Note that the direct image functor is left exact and so $N$ can be regarded as a submodule of $M$. Note that by Riemann's Hebbarkeissatz, every element in $M$ (and in N) is uniquely determined by its restriction to $U$. By restriction, we can consider $\psi_* f \in M(U)$ (we keep the same notation after restriction). We know that $\psi: \psi^{-1} (U) \rightarrow U$ is a finite, \'etale morphism of rank $n\in \mathbb{Z}_{\geq 1}$. By Kisin (see Prop. 1.3.2 and Cor. 1.3.3 in \cite{Kis99}), we know that there are isomorphism as $T_d$-modules $M (U) \cong M \langle \langle z^{-1}\rangle \rangle \cong M \otimes_{T_d} T_d \langle \langle z^{-1}\rangle \rangle$. Note that $ T_d \langle \langle z^{-1}\rangle \rangle$ is a flat $T_d$-algebra (see Cor. 1.3.3 in \cite{Kis99}). Kisin's theory of superconvergent modules ensures us that the natural map $M\rightarrow M\langle \langle z^{-1}\rangle \rangle$ is closed (see Lemma 1.3.7 in \cite{Kis99}). Hence, since $f$ is a global section, there exists a $\beta \in \mathbb{Z}_{\geq 1}$ such that $z^\beta \psi_* f = \sum b_i e_i$ where $b_i \in T_d$. The idea now is to set a linear system which will allow us to isolate the coefficients $b_i$ and estimate their norm. In order to do this, consider the set $\{y_1 , \dots, y_n\}$ of points in $X$ giving the fiber of $\psi$ at $x \in \mathbb{B}^d$. Evaluating $\psi_* f$ at the point $x$ gives us the following linear system: 
$$(*) \quad\quad\quad ( z(y_1 )^\beta f(y_1 ) , \dots, z(y_n )^\beta f(y_n )) =(b_1 (x) , \dots, b_n (x) )\cdot N_e ,$$
where $N_e$ is the matrix having generic term $(N_e )_{i, j} = e_j (y_i )$. Now, in order to isolate the coefficients $b_i$, we need to ensure that the matrix $N_e$ is invertible. To do that, we rely on the structure theorem of \'etale morphisms (we recall that we are now under restriction on $\psi^{-1}(U)$ where we know $\psi$ to be finite \'etale). To be precise, by the structure theorem of \'etale morphism (see for example in EGA IV Prop. 18.4.5 and 18.4.6 in \cite{Gro67}), we know that locally at any point $y\in \psi^{-1} (U)$, we have $A_y \cong (T_d)_{\psi(y)} [X] / (Q)$ for a certain polynomial $Q$ in the localization $(T_d)_{\psi(y)}$. As a consequence, after restriction to $U$, there exist $d_i \in \mathcal{O}_{\mathbb{B}^d} (U) \cong T_d \langle\langle z^{-1} \rangle \rangle$ such that:
$$\psi_* f =\sum^n_{i=1} d_i (X^i \mod Q) \quad \text{ in }M(U).$$
If we repeat the construction of the liner system $(*)$, it is clear that in the new basis $\{X^i \mod Q\}_{i=1, \dots, n}$, the correspondent matrix $N$ (which plays the role of $N_e$) is now a Van Der Monde matrix. Note that all the points in the fiber of $x \in U$ (i.e. the points $y_i \in \psi^{-1} (U)$) are distinct because $\psi$ is \'etale and so $N$ is invertible. Since $N_e$ can be obtained from $N$ via a base change, it follows that $N_e$ is invertible. \\
Finally, from the linear system $(*)$ it is clear that we can isolate the coefficients $b_i$ and bound them in terms of the norm of $f$, the norm of $z^\beta$ and the norm of the inverse of the matrix $N_e$, which can be written in terms of the norm of the adjoint matrix $N^{Ad}_e$ and the norm of the determinant of $N_e$. Since $f$ and $z$ are power-bounded, up to replacing $\beta$ by a bigger positive integer $\alpha$, we have that there exists $b_i \in T_d^\circ$ such that:
$$z^\alpha \psi_* f =\sum^n_{i=1} b_i e_i \quad \text{ in }M(U),$$
which completes the proof.
\end{proof}
\noindent
The next step is to prove the following: 
\begin{claim}
For all $q\geq 1$, there exists $\alpha\in \mathbb{Z}_{\geq 1}$ such that $h^\alpha H^q (X , \mathcal{O}_X^\circ )=0$, for $h=\psi(z) \in \mathcal{O}^\circ_X (X)$.
\end{claim}
\noindent
The idea is to use the fact that the integral cohomology of the closed affinoid adic ball $\mathbb{B}^d$ is known and then transport such information to $X$. To be precise, the following result is due to Bartenwerfer (see \cite{Bar78}, \cite{Bar79}): 
\begin{proposition}
We have that $H^q (\mathbb{B}^d , \mathcal{O}^\circ_{\mathbb{B}^d })=0$ for all $q>0$. 
\end{proposition}
\noindent
We have an exact sequence of sheaves on $\mathbb{B}^d$:
$$0 \rightarrow (\mathcal{O}^\circ_{\mathbb{B}^d})^{\oplus n} \rightarrow \psi_* \mathcal{O}^\circ_X \rightarrow \psi_* \mathcal{O}^\circ_X / (\mathcal{O}^\circ_{\mathbb{B}^d})^{\oplus n} \rightarrow 0.$$
Consider the long exact sequence in cohomology: 
$$\rightarrow H^q (\mathbb{B}^d , (\mathcal{O}^\circ_{\mathbb{B}^d})^{\oplus n}) \rightarrow H^q (\mathbb{B}^d , \psi_* \mathcal{O}^\circ_X )\rightarrow H^q (\mathbb{B}^d , \psi_* \mathcal{O}^\circ_X / (\mathcal{O}^\circ_{\mathbb{B}^d})^{\oplus n}) \rightarrow H^{q+1} (\mathbb{B}^d , (\mathcal{O}^\circ_{\mathbb{B}^d})^{\oplus n}) )\rightarrow$$
Now the first and last term are zero for all $q\geq 1$ because of Bartenwerfer result $H^q (\mathbb{B}^d , \mathcal{O}^\circ_{\mathbb{B}^d})=0$ for all $q\geq 1$ and the fact that sheaf cohomology commutes with finite direct sums. Note that the term $H^q (\mathbb{B}^d , \psi_* \mathcal{O}^\circ_X / (\mathcal{O}^\circ_{\mathbb{B}^d})^{\oplus n}) $ is killed by multiplication-by $ z^\alpha$, hence the same holds for $H^q (\mathbb{B}^d , \psi_* \mathcal{O}^\circ_X )$. However, we still cannot deduce directly that also the cohomology of $X$ is killed by a power of $\psi(z):=h$ because of the absence of a global base change argument for the cohomology. We rely again on Riemann's Hebbarkeissatz via a computation using local cohomology (see Exp. I and II in \cite{Gro68}, Grothendieck's SGA2). For any abelian sheaf $\mathcal{F}$ over a space $Y$ and for $Z$ closed inside $Y$ define the local cohomology $H^q_Z (Y , \mathcal{F})$ of $\mathcal{F}$ at $Z$ as the derived functor of $\text{Ker}(\mathcal{F} (Y) \rightarrow \mathcal{F} (Y -Z))$. The local cohomology satisfy the following long exact sequence (see Cor. 2.9 in \cite{Gro68} or also \cite{Kis99}):

$$\rightarrow H^q_Z (Y, \mathcal{F}) \rightarrow H^q (Y , \mathcal{F}) \rightarrow H^q (Y-Z , \mathcal{F}) \rightarrow H^{q+1}_Z (Y , \mathcal{F})\rightarrow$$ 
\noindent
Now apply the above in our situation, i.e. $Y=\mathbb{B}^d$ and $\mathcal{F}=\psi_* \mathcal{O}_{\mathbb{B}^d}^\circ$ and $Z=V$ we get:

$$\rightarrow H^q_V (\mathbb{B}^d ,  \psi_* \mathcal{O}_X^\circ) \rightarrow H^q (\mathbb{B}^d ,  \psi_* \mathcal{O}_X^\circ) \rightarrow H^q (\mathbb{B}^d-V ,  \psi_* \mathcal{O}_X^\circ) \rightarrow H^{q+1}_V (\mathbb{B}^d ,  \psi_* \mathcal{O}_X^\circ)\rightarrow$$ 
\noindent
Since $\psi : X \rightarrow \mathbb{B}^d$ is surjective and since by Riemann's Hebbarkeissatz $j_* \mathcal{O}^\circ_{X-Z} \cong \mathcal{O}_X^\circ$, we deduce that $H^q_V (\mathbb{B}^d , \psi_* \mathcal{O}_X^\circ )=0$ for all $q\geq 1$. Hence, we deduce that $H^q (\mathbb{B}^d , \psi_* \mathcal{O}_X^\circ )\cong H^q (\mathbb{B}^d -V , \psi_* \mathcal{O}_X^\circ )$. In a similar way for $X$, we deduce that $H^q (X , \mathcal{O}_X^\circ )\cong H^q (X -Z ,  \mathcal{O}_X^\circ )$.\\
To slightly simplify the notation, denote $\tilde{X} := X-Z$ and $\tilde{\mathbb{B}}^d := \mathbb{B}^d -V$. Note that $\tilde{X}$ and $\tilde{\mathbb{B}}^d$ are quasi-compact and quasi-separated analytic spaces but not affinoid. In particular, thanks to Prop. 2.5.4 of De Jong and Van der Put in \cite{DJVP96}, we can compute their sheaf cohomology via the Cech cohomology (property that holds in more general cases, see \cite{DJVP96}). From the above discussion, we know that $H^q (\tilde{\mathbb{B}}^d , \psi_* \mathcal{O}_X^\circ )$ is uniformly torsion for all $q \geq 1$. In order to relate the torsion information on these cohomology groups to torsion informations on the cohomology of $\tilde{X}$ we now rely on the full power of the following base change theorem for rigid analytic spaces of De Jong and Van der Put (see Theorem 2.7.4 in \cite{DJVP96}):

\begin{theorem}\label{stalk}
Let $f : X \rightarrow Y$ be a quasi-compact morphism of rigid analytic varieties. Take any analytic point $a\in Y$ and denote by $X_a$ the fiber of $f$ over $a$ (defined as fiber product). The functors $S \mapsto H^n (X_a , S|_{X_a})$ (resp. $S\mapsto (R^n f_* S )_a $) on the category of Abelian Sheaves on $X$ form a $\delta$-functor. These $\delta$-functors are isomorphic:
$$(R^n f_* S )_a \cong H^n (X_a , S|_{X_a}) \text{ for any Abelian Sheaf }S\text{ on }X.$$ 
\end{theorem}
\noindent
Recall that $S|_{X_a}$ simply denotes the pull back sheaf $\omega^* S$ via the natural projection morphism $\omega : X_a \rightarrow X$. 
Note indeed that while this theorem is particularly useful when we deal with overconvergent sheaf (which in some sense are determined by their stalks, see property 6 in Lemma 2.3.2 in \cite{DJVP96}), it holds for general abelian sheaves, which serves our purposes. 
Moreover, note the important point that if in the above theorem the abelian sheaf $S$ is in particular a $\mathcal{O}_Y$-module (resp. $\mathcal{O}_Y^+$-module) then both functors $(R^n f_* \bullet )_a$ and $H^n (X_a , \bullet|_{X_a})$ are naturally functors of $\mathcal{O}_{X, a}$-modules (resp. $\mathcal{O}_{X, a}^+$-module) for each fixed analytic point $a$ in $Y$ and the isomorphism holds in that category. \\
We now apply the above theorem to our situation, i.e. $X=\tilde{X}$, $Y=\tilde{\mathbb{B}}^d$, $f=\psi : \tilde{X}\rightarrow \tilde{\mathbb{B}}^d$ which is finite and \'etale and $\mathcal{F}=\mathcal{O}_{\tilde{X}}^\circ$ seen as functor of $A^\circ$-modules (note that this comes from the restriction maps and Riemann's Hebbarkeissatz which grants that $j_* \mathcal{O}^\circ_{\tilde{X}} \cong \mathcal{O}^\circ_X$). Since $H^q (\tilde{\mathbb{B}}^d , \psi_* \mathcal{O}_X^\circ )$ is killed by $z^\alpha$, we have that for any analytic point $a\in \tilde{\mathbb{B}}^d$ the $\mathcal{O}_{\tilde{\mathbb{B}}^d , a}$-modules $(R^q f_* \mathcal{O}_X^\circ )_a $ are also killed by $z^\alpha$ because the cohomology of $\tilde{\mathbb{B}}^d$ is isomorphic to its Cech cohomology. By Theorem \ref{stalk}, we deduce that $H^q (X_a , {\mathcal{O}_X^\circ}|_{X_a}) $ is also killed by $z^\alpha$. Now, since the morphism $\psi : \tilde{X}\rightarrow \tilde{\mathbb{B}}^d$ is finite and \'etale we can decompose $X_a$ as a disjoint union of $X_{b_i}$ for the analytic points $b_i \in X$ such that $\psi(b_i )=a$ for $i=1, \dots, n$. We deduce that for all analytic points $c\in X$ we have that $H^q (X_c , \mathcal{O}^\circ_{X} |_{X_c} )$ is killed by $h^\alpha$ (with $\psi(z)=h$, note that $\psi$ is injective on global sections). As observed before, the cohomology of $\tilde{X}$ is isomorphic to the Cech cohomolgy and hence we deduce that $H^q (\tilde{X}, \mathcal{O}_X^\circ)$ is killed by $h^\alpha$ thanks to the Key Lemma of De Jong and Van der Put (see Lemma 2.7.1 in \cite{DJVP96}). Finally, thanks to the Riemann's Hebbarkeitssatz, we conclude that $H^q (X , \mathcal{O}_X^\circ )\cong H^q (\tilde{X} ,  \mathcal{O}_X^\circ )$ is killed by $h^\alpha$.\\ 
The final step consists of proving that there exists $N\in \mathbb{Z}_{\geq 1} $ such that $\mathfrak{m}^N$ kills $H^q (X , \mathcal{O}_X^\circ )$. Consider the ideal $J$ inside $A^\circ:=\mathcal{O}^\circ_X (X)$ generated by $h^\alpha$. Since the sets of the form $c A^\circ $ with $0\not=c \in \mathbb{K}^\circ$ form a fundamental system of neighborhood for the topology of $A$, we know that there exists a $N\in\mathbb{Z}_{\geq 1}$ such that $(\pi^N ) \subset (h^\alpha )$, where $\pi$ is a fixed pseudouniformizer for $\mathbb{K}$. This implies that there exists an $f\in \mathcal{O}_{X}^\circ$ such that $\pi^N = f \cdot h^\alpha$ which allows us to conclude that  $H^q (X , \mathcal{O}_X^\circ )$ is killed by $\pi^N$ and it is also uniformly torsion in the sense of Hansen and Kedlaya (see \cite{HK25}). 
\end{section}

 
\printbibliography[heading=bibintoc,title={References}]

\end{document}